\documentclass[11pt,a4paper]{amsart}
\usepackage{amsthm,amsmath,amssymb}
\usepackage[english]{babel}
\usepackage{graphicx}
\usepackage{hyperref}
\usepackage[foot]{amsaddr}

\newtheorem{thm}{Theorem}
\newtheorem*{thm*}{Theorem}
\newtheorem{lem}{Lemma}
\newtheorem*{lem*}{Lemma}
\newtheorem*{problem*}{Problem}
\newtheorem{definition}{Definition}

\title[Finsler metrics admitting three projective vector fields]{Finsler metrics on surfaces admitting three projective vector fields}
\author{Julius Lang}
\address{Affilation and adresses of the author: Friedrich-Schiller University Jena, FMI\\
 Ernst-Abbe-Platz 2, 07743 Jena, Germany\\
 julius.lang@uni-jena.de}


\begin{document}

	\begin{abstract}
	We show that in dimension 2 every Finsler metric with at least $3$-dimensional Lie algebra of projective vector fields is locally projectively equivalent to a Randers metric. We give a short list of such Finsler metrics which is complete up to coordinate change and projective equivalence.
	\end{abstract}
	
	\maketitle

\section{Introduction}
	Two Finsler metrics on a surface $M^2$ are called \textit{projectively equivalent}, if their geodesics coincide as oriented point sets. A vector field on $M$ is called \textit{projective} for a Finsler metric $F$, if its flow preserves the projective class of $F$, i.e. it takes geodesics to geodesics as point sets.
	The projective vector fields form a finite-dimensional Lie algebra $\mathfrak p(F)$ and it follows from classical results that on a $2$-dimensional manifold, if $\dim \mathfrak p(F)>3$, then $\dim \mathfrak p (F)=8$ and $F$ is projectively flat, see \cite{LieODEs,Tresse2,Tresse}.
	
	In 1882 Sophus Lie \cite{lie} stated the following problem: 
	\begin{problem*}
	Describe metrics on surfaces with $\dim \mathfrak p \geq 2$.
	\end{problem*}
	The original problem intended for pseudo-Riemannian metrics was solved in \cite{matveev}, where local normal forms for metrics with $\dim \mathfrak p \geq 2$ were given. It is interesting to study the generalization of this problem for Finsler metrics, where many additional examples appear.
	
	Finsler metrics with $\dim \mathfrak p =8$ are exactly the projectively flat metrics, i.e. for which one can find local coordinates in which all geodesics have straight trajectories. The investigation of such metrics was stimulated by being the 4th of the Hilbert problems and there are numerous results present in the literature, e.g. \cite{AlvarezProfFlat1, AlvarezProfFlat2,ShenProjFlat3,ShenProjFlat1,ShenProjFlat2} - see \cite{overviewFlat} or \cite[Chapter 6.2]{ShenModernFinsler} for an overview.
	
	The main result is the theorem below and gives an answer to the problem in the submaximal case $\dim \mathfrak p =3$.

	\begin{thm}\label{theorem1}
	Every Finsler metric on a surface admitting at least three independent projective vector fields is projectively equivalent near any transitive point to one of the following:
		\begin{itemize}
			\item a \textit{Randers metric} $F= {\alpha} + \beta$, where $\alpha$ is of constant sectional curvature and $\mathfrak p(F)$ is the Killing algebra of $\alpha$.
			\item a \textit{Riemannian metric}.
		\end{itemize}
		In some local coordinates $F$ is projectively equivalent to the Euclidean metric or to one of the following:
	$$
	\begin{matrix}		
	(a)& \sqrt{dx^2+dy^2}  + \tfrac12(y dx -x dy) &&
	(b_k^\pm) & \frac{\sqrt{dx^2+dy^2} - \tfrac k2(y dx -x dy)}{1\pm(x^2+y^2)}, k >0 \\
	(c^+)& ~\sqrt{ \frac{e^{3x}}{(2 e^x - 1)^2}dx^2 + \frac{e^x}{2e^x-1} dy^2  }&&
	(c^-)& ~\sqrt{ e^{3x}dx^2 + e^x dy^2}
	\end{matrix}	
	$$
	Each of this metrics is strictly convex on a neighborhood of the origin and none of them is locally isometric to any Finsler metric projectively equivalent to one of the others.
	\end{thm}
	
	\textbf{Setup.} Throughout the paper we work locally on a $2$-dimensional manifold $M$ around a point $o \in M$. All objects are assumed to be $C^\infty$-smooth and defined locally on $M$, i.e. on a coordinate neighborhood $U$ with coordinates $(x,y)$, but defined fiber-globally over $U$, i.e. on the whole of $TU=U\times \mathbb R^2$. The natural fiber coordinates on $T_{(x,y)}U$ will be denoted by $(u,v)$. We work around a point where the projective vector fields are transitive, i.e. they span the whole tangent plane at this point.
	
	\textbf{Structure of the proof.} A convenient way to formalize the system of geodesics of a Finsler metric is as a vector field on $TM$, called a \textit{spray} (for exact definitions, see below), so that the geodesics are the projection of its integral curves to $M$.	 	
	The portion of curves of a spray for which $\dot x >0$ and the portion with $\dot x <0$ may be described by a second order ODE $\ddot y(x)=f_{\pm}(x,y(x),\dot y(x))$ each and coincide for projectively equivalent sprays.  These \textit{induced ODEs} are the main tool for our classification of sprays with projective symmetries, due to the fact that every projective vector field preserves the induces ODEs. They were studied for affine connections at least since the time of Beltrami and are sometimes referred to as \textit{projective connection}.
	
	In section \ref{SectionSprayNormalForms} we deduce a list of normal forms of such ODEs preserved by three independent vector fields. The techniques to obtain this list were already described more than 100 years ago by Sophus Lie \cite{LieODEs}. 
	
	\begin{lem}\label{lemma1} Let $\ddot y=f(x,y(x),\dot y(x))$ be a second order ODE and $X_1,X_2,X_3$ three linearly indepent vector fields on the plane, whose flow preserves the ODE. Then in some local coordinates around $0$, the equation takes the form $\ddot y=0$ or one of the following:
	$$
	\begin{matrix}	
	D1 & \ddot y = C(y^2-2\dot y)^{3/2}-y^3+3y\dot y\\
D2 & \ddot y = C\dot y^{\frac{\lambda-2}{\lambda-1}}\\
J1 & \ddot y = C \dot y^3 e^{-1/\dot y}\\
J2 & \ddot y = \frac{1}{2}\dot y + Ce^{-2x}\dot y^3\\
C1 & \ddot y = C(\dot y^2+1)^{3/2}e^{-\lambda \arctan(\dot y)}\\
C2 & \ddot y = \frac{C (\dot y^2+1)^{3/2} \pm 2(x\dot y-y)(\dot y^2+1)}{1\pm(x^2+y^2)}\\
	\end{matrix}	
	$$
	\end{lem}
	
	Not all of these ODEs can describe the portion of curves with $\dot x >0$ of a \textit{fiber-globally} defined spray. By sorting out those, we obtain normal forms for sprays with three independent projective vector fields up to projective equivalence.
	
	\begin{lem}\label{lemma2}Let $\Gamma$ be a spray on the plane with $\dim \mathfrak p \geq 3$. Then there are local coordinates in which $\Gamma$ is projectively flat or projectively equivalent to one of:

$$
\begin{matrix}

(a) & u\partial_x + v \partial_y - \sqrt{u^2+v^2}(v \partial_u - u \partial_v)\\

			(b^\pm_k) & u \partial_x + v \partial_y
			-\frac{k \sqrt{u^2+v^2} \pm 2(yu-xv)}{1 \pm(x^2+y^2)} (v\partial_u - u \partial_v) & k>0\\
	
	(c^\pm) &  u \partial_x + v \partial_y -\tfrac{1}{2}( 3u^2 \pm e^{-2x}v^2)\partial_u -uv \partial_v
\end{matrix}	
	$$
	None of these sprays can be transformed into one projectively equivalent to one of the others by a local coordinate change.
	\end{lem}
	
	The sprays $(a)$ and $(b_k^{\pm})$ are very geometric: The curves of the spray $(a)$ are positively oriented circles of radius $1$ in the Euclidean plane. Similarly, the curves of $(b^+_k)$ and $(b^-_k)$ are the positvely oriented curves of constant geodesic curvature $k$ on the two-sphere $S^2$ in stereographic coordinates and in the Poincare disk model of the hyperbolic plane.
	%The sprays $(c^{\pm})$ seem to be less geometric - however in adjusted coordinates their trajectories are the images of $S^1 \subseteq \mathbb R^2$ (the standard hyperbola respectively) under the standard action of $SL(2)$.
	
	We remark that the sprays $(c^{\pm})$ are \textit{geodesically reversible}, meaning that the unique geodesics tangent to the vectors $v$ and $-v$ have the same trajectories on $M$. The sprays $(a)$ and $(b_k^\pm)$ are \textit{geodesically irreversible}.
	
	In section \ref{SectionConstruction} we explain how one can calculate the induced ODE of the geodesic spray of a Finsler metric directly. This allows to check quickly that the geodesic sprays of the metrics from Theorem \ref{theorem1} are projectively equivalent to the sprays from Lemma \ref{lemma2}. This finishes the proof of Theorem \ref{theorem1}.
	
	Additionally we explain how to construct the metrics from Theorem \ref{theorem1}.   
	For the sprays $(c^\pm)$ we give Riemannian metrics; for the sprays $(a)$ and $(b_k^\pm)$ we construct so called \textit{Randers metrics} by adding an appropriate 1-form to the Riemannian metrics $\alpha$ of constant sectional curvature.

	\textbf{Rigidity.} Theorem \ref{theorem1} rises the question how rigid the found Finsler metrics are - in other words, are there Finsler metrics projectively equivalent to the ones we constructed?
	
	There is always a trivial freedom of scaling and adding a closed 1-form: Suppose two Finsler metrics $F,\tilde F$ are related by $\tilde F = c F + \beta$  for some $c > 0$ and a 1-form $\beta$ with $d \beta =0$, then they are projectively equivalent.
	For \textit{geodesically reversible} Finsler metrics it is not hard to see that there are many non-trivially projectively equivalent Finsler metrics: at least for every function on the space of unoriented geodesics one can construct a non-trivial Finsler metric projectively equivalent to the original one, see \cite{alvarez,Arcostanzo}.
	
	For Finsler metrics with irreversible geodesics the situation is quite different: already for the Finsler metric $(a)$ it is not easy to find a non-trivially projectively equivalent Finsler metric. In \cite{randersRigidity} it was proven that two Randers metrics are projectively equivalent if and only if they are trivially related.
	However it was noted by S. Tabachnikov \cite{Tabachnikov} that every smooth measure on the plane (seen as the space of geodesics) such that every ball of radius $1$ has measure $1$, gives rise to a non-trivial Finsler metric projectively equivalent to $(a)$. The question weither such non-constant measures exist is the so called Pompeiu problem and has a positive answer, hence giving rise to new Finsler metrics projectively equivalent to $(a)$.

\section{Sprays with $3$-dimensional projective algebra up to projective equivalence}\label{classificationGeodesics}\label{SectionSprayNormalForms}

In this section we describe sprays with at least $3$-dimensional projective algebra up to coordinate change and projective equivalence by prooving Lemma \ref{lemma2}. The result is a direct consequence of Lemma \ref{lemma1} - the classification of second order ODEs admitting three independent infinitesimal point symmetries. To obtain this, we first describe all $3$-dimensional Lie algebras of vector fields in the plane up to coordinate change around a transitive point and determine for each, which ODEs admit the prescribed vector fields as infinitesimal point symmetries.

\subsection{$3$-dimensional Lie algebras of vector fields in the plane}\label{sectionLieAlgebrasVectorFields}
Let $\mathfrak g, \tilde {\mathfrak g}$ be $3$-dimensional Lie algebras of vector fields on the plane
 and $\mathfrak g_{0} = \{ X \in \mathfrak g \mid X|_0 = 0\}$ be the \textit{isotropy subalgebra} at the origin.
 We assume that both $\mathfrak g, \tilde {\mathfrak g}$ are \textit{transitive} at the origin, i.e. $\dim \mathfrak g_0 = \dim \tilde{\mathfrak g}_0=1$ and that no $X \not=0$ vanishes on an entire open neighborhood.
\begin{lem}\label{lemmaRealizationOfLieAlgebras}
\begin{enumerate}
\item There is a local coordinate transformation $\varphi$ taking each vector field from $\mathfrak g$ to one from $\tilde{\mathfrak g}$ and fixing the origin, if and only if there is an isomorphism of Lie algebras $\psi: \mathfrak g \to \tilde{\mathfrak g}$ which takes $\mathfrak g_0$ to $\tilde{\mathfrak g}_0$.
\item $\mathfrak g_0$ is not an ideal, i.e. there is a $X \in \mathfrak g$ with $[X,\mathfrak g_0] \not\subseteq \mathfrak g_0$.
\end{enumerate}
\end{lem}
\begin{proof}
If $\varphi: \mathbb R^2 \to \mathbb R^2$ is a diffeomorphism with $\varphi(0)=0$ that takes the vector fields from $\mathfrak g$ to vector fields from $\tilde{\mathfrak g}$, then $\psi: \mathfrak g \to \tilde{\mathfrak g}, X \mapsto \varphi_* X$ is an isomorphism of Lie algebras, and $(\varphi_* X)|_0 = 0$ if and only if $X|_0 = 0$. Hence $\psi(\mathfrak g_0) = \tilde{\mathfrak g}_0$.

For the other direction, let $X_0 \in \mathfrak g_0$ and $X_1,X_2 \in \mathfrak g$ form a basis of $\mathfrak g$ and set $\tilde X_i = \psi(X_i)$ for $i=1,2,3$.
	Let $\rho: \mathbb R^2 \to \mathbb R^2$ be the local diffeomorphism defined by $(s,t) \mapsto \Phi_{X_1}^s \circ \Phi_{X_2}^t(0)$ and $\tilde \rho$ accordingly, where $\Phi_X$ denotes the flow along $X$. We show that for $i=1,2,3$ the vector fields $Y_i := \rho^*(X_i)$ and $\tilde Y_i := \tilde \rho^*(\tilde X_i)$ coincide, so that $\phi:=\tilde \rho \circ \rho^{-1}$ is the sought coordinate transformation.
	
	It is obvious that $Y_1 = \partial_s$, $Y_2|_{(0,t)}=\partial_t$ and $Y_0|_{(0,0)}=0$. Let $C^k_{ij}$ be the structure constants of the basis $X_0,X_1,X_2$, so that $[Y_i,Y_j]=C^k_{ij}Y_k$.
	For points $(0,t)$ consider the equations $[Y_0,Y_2] = C^k_{02}Y_k$. The only unknowns in the right side are the components of $Y_0$. In the left side, we can replace all derivatives by $s$ in terms of components of the $Y_i$'s using the commutation relations with $Y_1=\partial_s$. Since $Y_2|_{(0,t)}=\partial_t$, in each equation exactly one derivative by $t$ survives with coefficient $1$. Hence we have a system of two ODEs on the components of $Y_0$ in points $(0,t)$ with starting value $Y_0|_{(0,0)} = 0$, so that the vectors $Y_0|_{(0,t)}$ are uniquely defined by the structure constants. Now for fixed $t_0$ the four equations $[Y_0,\partial_s]=C^k_{01}Y_k, [\partial_s,Y_2]=C^k_{12}Y_k$ in points $(s,t_0)$ again form a system of ODEs with already determined starting values, so that $Y_0,Y_1,Y_2$ are determined on a neighborhood of the origin only by the structure constants. But the same holds for $\tilde Y_0,\tilde Y_1,\tilde Y_2$, since they have the same structure constants, so we have $\tilde Y_i=Y_i$.
	
	For the second statement, suppose that $X_0$ commutes with $X_1,X_2$. Then the first system of ODEs $[Y_0,Y_2]=0$ in points $(0,t)$ has the obvious solution $Y_0|_{(0,t)} \equiv 0$ and by $[Y_0,Y_1]= 0$ we would have $Y \equiv 0$ on a neighborhood of the origin.	
\end{proof}

Lemma \ref{lemmaRealizationOfLieAlgebras} explains how we can obtain a complete list of $3$-dimensional Lie algebras of vector fields around a transitive point up to coordinate transformation: For every pair $(\mathfrak g, \mathfrak h)$ of an abstract $3$-dimensional Lie algebra and $1$-dimensional subalgebra, which is not an ideal, we should find one representative Lie algebra of vector fields isomorphic to $\mathfrak g$ such that the isotropy subalgebra corresponds to $\mathfrak h$.

Let $(\mathfrak g, \mathfrak h)$ be fixed and $X_0 \in \mathfrak h \backslash\{0\}$.  We might choose $X_1,X_2$ such that  the map $\operatorname{ad} X_0 : \mathfrak g \to \mathfrak g, X \mapsto [X_0,X]$ is given in the basis $X_0,X_1,X_2$ by a matrix
$$\operatorname{ad} X_0
= \begin{bmatrix} 0 & *  \\ 0 & A 

\end{bmatrix},
$$
where $A \not=0$ is a $(2 \times 2)$ Jordan block, and by scaling $X_0$ we might scale this matrix by any nonzero constant.

For the diagonal case $A=\begin{bmatrix}
1 & 0 \\ 0 & \lambda
\end{bmatrix}$ we can restrict to $|\lambda| \geq 1$. For the Jordan case $A=\begin{bmatrix}
\lambda & 1 \\ 0 & \lambda
\end{bmatrix}$ we only need to consider $\lambda=0$ and $\lambda = 1$. For the complex case $A=\begin{bmatrix}
\lambda & 1 \\ -1 & \lambda
\end{bmatrix}$ we can restrict to $\lambda \geq 0$.

Let $[X_0,X_1]=\sum_{i=0}^2\alpha_i X_i$,$[X_0,X_2]=\sum_{i=0}^2\beta_i X_i$ and $[X_1,X_2]=\sum_{i=0}^2\gamma_i X_i$. The Jacobi identity $[X_0,[X_1,X_2]]+[X_1,[X_2,X_0]]+[X_2,[X_0,X_1]]=0$ is given by

$$\begin{matrix}
\alpha_0\gamma_1 + \beta_0\gamma_2 - \beta_2 \gamma_0 -\alpha_1 \gamma_0=0\\
 \beta_1 \gamma_2 + \alpha_1 \beta_0 - \beta_2 \gamma_1 - \alpha_0 \beta_1  =0\\
\alpha_2 \gamma_1 + \alpha_2 \beta_0 - \alpha_0 \beta_2 - \alpha_1 \gamma_2 =0
\end{matrix}.$$

In the diagonal case, the Lie bracket table is of the form 
$$\begin{bmatrix}
\alpha_0 X_0 + X_1 & \beta_0 X_0 + \lambda X_2\\ & \gamma_0X_0 +\gamma_1X_1 + \gamma_2X_2& 
\end{bmatrix}$$
By adding a multiple of $X_0$ to $X_1$, we can assume $\alpha_0=0$. By the Jacobi identity we get 
$(1+\lambda)\gamma_0=0$, $\beta_0=\lambda \gamma_1$ and $\gamma_2=0$. If $\lambda \not=0$, we might add a multiple of $X_0$ to $X_2$ to assume $\beta_0=0$, so that we can assume this in any case. If $\gamma_0 \not=0$, we have $\lambda=-1$ and $\gamma_1=0$ and by scaling $X_1$ we can obtain $\gamma_0=1$.
If $\gamma_0=0$ either $\gamma_1=0$ holds already (if $\lambda\not=0$), or replacing $X_2$ by $X_2+\gamma_1 X_0$ achieves this without changing the other relations (if $\lambda=0$). Hence there are only two tables that we need to consider:
$$\begin{bmatrix}X_1 & - X_2 \\ & X_0\end{bmatrix}
\begin{bmatrix}
X_1 & \lambda X_2 \\ & 0
\end{bmatrix}
$$
Each describes an equivalence class of a pair $(\mathfrak g, \mathfrak h)$. To find corresponding Lie algebras of vector fields one could choose $X_1$ arbitary and then solve ODEs follow the proof of Lemma \ref{lemmaRealizationOfLieAlgebras}. However we will just present a solution for every pair having the correct Lie bracket table.
$$\begin{matrix}
D1 & -x \partial_x + y \partial_y & \partial_x & -\frac{1}{2}x^2 \partial_x+(xy+1)\partial_y\\
D2 & -x \partial_x - \lambda y\partial_y & \partial_x & \partial_y 
\end{matrix}
$$

In the Jordan case, the Lie bracket table is of the form 
$$\begin{bmatrix}
\alpha_0 X_0 + \lambda X_1  & \beta_0 X_0 +X_1+ \lambda X_2\\ & \gamma_0X_0 +\gamma_1X_1 + \gamma_2X_2& 
\end{bmatrix}$$
If $\lambda=1$, again by adding multiples of $X_0$ to $X_1$ and $X_2$, we can assume $\alpha_0=\beta_0=0$. From the Jacobi identity we get $\gamma_0=\gamma_1=\gamma_2=0$.
If $\lambda =0$, we can assume $\beta_0=0$ and get from the Jacobi identity $\alpha_0=\gamma_2$ and $\gamma_1 \gamma_2=0$. In the case that $\gamma_2 \not=0$, we have $\gamma_1=0$ and can assume $\gamma_2=1$ by scaling $X_1$ and $X_2$ by $\frac{1}{\gamma_2}$. Then by replacing $X_2$ by $X_2+\frac{\gamma_0}{2}X_0$, we can also assume $\gamma_0=0$. The case that $\gamma_2=0$, also $\alpha_0=0$ and we could assume $\gamma_0,\gamma_1 \in \{0,1\}$, but we can give a Lie algebra of vector fields with parameters that covers all cases. Hence we deal with three cases:

$$\begin{bmatrix}
X_1 & X_1+X_2 \\ & 0
\end{bmatrix}
\begin{bmatrix}
X_0 & X_1 \\ & X_2
\end{bmatrix}
\begin{bmatrix}
0 & X_1 \\ & \gamma_0 X_0 + \gamma_1 X_1
\end{bmatrix}
$$
$$
\begin{matrix}
J1 & -(x+y) \partial_x -y\partial_y & \partial_x & \partial_y\\
J2 & -y \partial_x - \tfrac12 y^2 \partial_y & -\partial_x - y \partial_y & -\partial_y & (c^\pm)\\
J3 & y \partial_x & \partial_x & (\gamma_0y+\gamma_1)x\partial_x+(\gamma_0y^2+\gamma_1y-1)\partial_y
\end{matrix}
$$

In the complex case the Lie bracket table is of the form 
$$\begin{bmatrix}
\alpha_0 X_0 +  \lambda X_1 -X_2  & \beta_0 X_0 +X_1+ \lambda X_2\\ & \gamma_0X_0 +\gamma_1X_1 + \gamma_2X_2& 
\end{bmatrix}$$
By replacing $X_1$ by $X_1 + r X_0$ and $X_2$ by $X_2 + s X_0$, the conditions that the new coefficients of $X_0$ in the first row become vanish the linear equations $\lambda r + s = \alpha_0$ and $r + \lambda s$, which admit a solution. Hence we can assume $\alpha_0=\beta_0=0$. The Jacobi identity gives $\lambda \gamma_0=0$ and $\gamma_1=\gamma_2=0$. If $\gamma_0=0$, we are left with $\lambda \geq 0$ as a parameter. Otherwise $\lambda=0$ and by scaling $X_1,X_2$ by $\frac{1}{\sqrt{\gamma_0}}$ we can achieve $\gamma_0 = \pm 1$. Hence we have two cases:

$$\begin{bmatrix}
\lambda X_1-X_2 & X_1 + \lambda X_2\\ & 0
\end{bmatrix}
\begin{bmatrix}
-X_2 & X_1 \\ & \pm X_0
\end{bmatrix}
$$

$$\begin{matrix}
C1 & -(\lambda x-y) \partial_x - (x+\lambda y) \partial_y & \partial_x & -\partial_y & (a)\\
C2 & y \partial_x - x \partial_y & \frac{1}{2}(x^2-y^2\pm 1)\partial_x + xy \partial_y & xy \partial_x + \frac{1}{2}(-x^2+y^2\pm1) & (b_k^\pm)
\end{matrix}$$

We have shown that every $3$-dimensional Lie algebra of vector fields in the plane around a transitive point there are coordinates in which it is given by one of the seven types above.

\subsection{Second order ODEs with three independent infinitesimal point symmetries and proof of Lemma \ref{lemma1}}
A vector field $X$ on $\mathbb R^2$ is called \textit{infinitesimal point symmetry} of an ODE $\ddot y(x)=f(x,y(x),\dot y(x))$, if its local flow preserves the equation, or equivalently, it takes each trajectory of a solution $(x,y(x))$ to some trajectory of a solution. It was shown by Sophus Lie by that the infinitesimal point symmetries form a Lie algebra of dimension at most $8$ \cite[Chapter 17 \S 2,3]{LieODEs}.

	With a second order ODE one has associated a $1$-dimensional distribution $\langle D \rangle$ on the space $J \mathbb R^2=\mathbb R^2 \times \mathbb R$ of tangent directions not parallel to $\partial_y$, induced by the vector field $D|_{(x,y,z)}=\partial_x + z \partial_y + f(x,y,z)\partial_z$, where $(x,y,z)$ are the coordinates for the direction $\operatorname{span}(\partial_x+z\partial_y)$ in the point $(x,y)$.
	
	A curve $c: I \to \mathbb R^2$ is prolonged naturally to $J\mathbb R^2$ by $\hat c(t):=(c^1(t),c^2(t),\frac{\dot c^2(t)}{\dot c^1(t)})$. If the prolongation of a curve is tangent to $\langle D \rangle$, then the same is true for any reparametrization. Moreover the lift of a curve $c(x)=(x,y(x))$ is tangent to $\langle D \rangle$, if and only if $y(x)$ is a solution to $\ddot y = f(x,y(x),\dot y(x))$.
	
	If $X=a(x,y) \partial_x + b(x,y) \partial_y$ is a vector field on the plane, its induced flow on $J \mathbb R^2$ is generated by the vector field $\hat X|_{(x,y,z)}=a(x,y) \partial_x + b(x,y) \partial_y + c(x,y,z) \partial_z$, where $c=b_x+zb_y-z(a_x+za_y)$. Here the subscripts denote partial derivatives and arguments are supressed.

\begin{lem}\label{lemmaInfinitesimalPointsymmetry}
	\begin{enumerate}
	\item A vector field $X=a(x,y) \partial_x + b(x,y) \partial_y$ is an infinitesimal point symmetry of $\ddot y = f(x,y,\dot y)$, if and only if $[\hat X,D]$ is proportional to $D$ in every point $(x,y,z)$. This condition is given by
	\begin{equation}\label{pointSymmetryConditionEquation}
	a f_x + b f_y +c f_z = (c_z-a_x-za_y)f+c_x+zc_y.
\end{equation}		 
	 
	 \item For a prescribed $3$-dimensional algebra of infinitesimal point symmetries transitive at the origin, the function $f$ describing an ODE is determined uniquely by an initial value.
	\end{enumerate}
\end{lem}
\begin{proof}
	The vector field $X$ is an infinitesimal symmetry, if and only if its flow on $J \mathbb R^2$ preserves the distribution $\langle D \rangle$, i.e. if $\mathcal L_{\hat X} D = [\hat X, D]$ is a multiple of $D$. By direct calculation
	$$[\hat X,D]=-(a_x+za_y) \partial_x- z(a_x+za_y)\partial_y + (af_x+bf_y+cf_z-c_x - z c_y - f c_z)\partial_z$$
	can only be a multiple of $D$ with factor $-(a_x+za_y)$, which is the case if and only if
	$$af_x+bf_y+cf_z-c_x - z c_y - f c_z = -(a_x+za_y)f.$$
	Let $X_i = a^i \partial_x + b^i \partial_y$ for $i=0,1,2$. If the ODE $\ddot y = f(x,y,\dot y)$ admits all three as infinitesimal point symmetries, it must solve the corresponding three equations (\ref{pointSymmetryConditionEquation}). If the matrix $\begin{pmatrix}
	a^0 & b^0 & c^0 \\
	a^1 & b^1 & c^1 \\
	a^2 & b^2 & c^2
	\end{pmatrix}$ is regular in some point $(0,0,z_0)$, this system might be written in normal form in a neighborhood and has a unique solution. Otherwise the matrix is singular in all points $(0,0,z)$, and we must have $c^0(0,0,z) \equiv 0$, which gives $b^0_x(0,0)=a^0_y(0,0)=0$ and $a^0_x(0,0)=b^0_y(0,0)$. It follows that $[X_0,X]=a^0_x(0,0)X$ for $X=X_1,X_2$. It follows that its Lie algebra structure must be isomorphic to case $D2$ with $\lambda=1$ and the ODEs (\ref{pointSymmetryConditionEquation}) are given as $f_x=f_y=0$ and $f =0$.
\end{proof}

We need to cite one important result (\cite[Theorem 2, Proposition 1]{romanovski}):
\begin{lem}\label{lemmaWhenFlat}
 If the algebra of point symmetries of an ODE $\ddot y = f(x,y,\dot y)$ is more than $3$-dimensional, then it is $8$-dimensional and in some local coordinates the ODE is given as $\ddot y = 0$. This is the case if and only if it is of the form
 $$y''=A+B~y'+C~(y')^2+D ~(y')^3$$
	with functions $A,B,C,D$ depending just on $x,y$ satisfying
	$$
	\begin{matrix}
 -A_{yy} + \frac23 B_{xy}-\frac13 C_{xx}-DA_x - 2AD_x+CA_y+AC_y + \frac13 BC_x - \frac23 BB_y&=0\\
	 \frac23 C_{xy}-\frac13 B_{yy}-D_{xx}+AD_y +2DA_y-DB_x-BD_x - \frac13 C B_y + \frac23 CC_x&=0
	\end{matrix}.	
	$$
\end{lem}

We are now ready to proof Lemma \ref{lemma1}: For each $3$-dimensional algebra of vector fields from the last section, we check which ODE admits it as infinitesimal point symmetries. In each case one has a three equations (\ref{pointSymmetryConditionEquation}) that determine $f$ up to a constant by Lemma \ref{lemmaInfinitesimalPointsymmetry} and the system can be solved by elementary methods. The solutions are given below, where $C \not=0$ is a constant (for $C=0$, all ODEs fulfil the assumptions of Lemma \ref{lemmaWhenFlat}):
$$\begin{matrix} & f(x,y,z)\\
D1 & C(y^2-2z)^{3/2}-y^3+3yz\\
D2 & Cz^{\frac{\lambda-2}{\lambda-1}}\\
J1 & C z^3 e^{-1/z}\\
J2 & \frac{1}{2}z + Ce^{-2x}z^3\\
%J3 & h(y)z^3\\
%C \frac{z^3}{(\gamma_0 y^2+\gamma_1 y -1)^{3/2}}e^{-\frac{\gamma_1 \operatorname{arctanh}(\frac{2\gamma_0y+\gamma_1}{\sqrt{\gamma_1+4\gamma_0}})}{\sqrt{\gamma_1+4\gamma_0}}}\\
C1 & C(z^2+1)^{3/2}e^{-\lambda \arctan(z)}\\
C2 & \frac{C (z^2+1)^{3/2} \pm 2(xz-y)(z^2+1)}{1\pm(x^2+y^2)}
\end{matrix}$$

For $J3$ the equations for $X_0$ and $X_1$ force the equation to be of the form $h(y)z^3$ for some function $h$ and by Lemma \ref{lemmaWhenFlat} its algebra of infinitesimal point symmetries is $8$-dimensional. All the others do not fulfil the assumptions from Lemma \ref{lemmaWhenFlat} and hence their algebra of infinitesimal point symmetries is exactly $3$-dimensional.

\subsection{Proof of Lemma \ref{lemma2}}

A \textit{spray} $\Gamma$ on a smooth manifold $M$ is a vector field $\Gamma \in \mathfrak X(TM \backslash 0)$, where $TM \backslash 0$ is the tangent bundle with the origins removed, which in all local coordinates $(x^i, \xi^i)$ on $TM$ is given as $\Gamma = \xi^i \partial_{x^i} - 2G^i(x^i, \xi^i) \partial_{\xi^i}$, where the $G^i$ are smooth and positively $2$-homogeneous in $\xi$.
	Two sprays on $M$ are \textit{projectively equivalent} if the projections to $M$ of their integral curves, called \textit{geodesics}, coincide as oriented point sets. This is the case if and only if $\Gamma -\tilde \Gamma = \lambda(x,\xi) (\xi^i \partial_{\xi^i})$ for some function $\lambda$. 
	
	A vector field $X$ on $M$ is called \textit{projective} for a spray $\Gamma$ if its flow $\Phi_t^X$ maps geodesics to geodesics as point sets, that is if $\Gamma$ and $(\Phi_t^{\hat X})_* \Gamma$ are projectively related, where $\hat X$ is the lift of $X$ to $TM$. This is the case if and only if $\mathcal L_{\hat X} \Gamma = \lambda(x,\xi) (\xi^i \partial_{\xi^i})$ for some function $\lambda$ and it follows by the Jacobi identity that the projective vector fields form a Lie algebra $\mathfrak p (\Gamma)$.

	In this section we give a complete list of projective classes of sprays on $\mathbb R^2$ with $\dim \mathfrak p =3$ up to local coordinate change. Let $(x,y)$ be coordinates on $\mathbb R^2$ and $(x,y,u,v)$ the induced coordinates on $T \mathbb R^2$. To each projective class of sprays, we associate the two second order ODEs, whose solutions $y(x)$ are the reparametrizations by the parameter $x$ of the geodesics with $\dot x >0$ and $\dot x < 0$ respectively. This is independent of the choice of a representative $\Gamma$.
	
	Let us determine the induced ODEs in terms of the spray coefficients $G^i$.
	Let $(x,y(x))$ be a curve, s.t. $(\varphi(t), y(\varphi(t)))$ is a geodesic for $\Gamma$ for a re\-parametrization $\varphi: \mathbb R \to \mathbb R$. Then
	$$
	\begin{array}{rrl}
	&\varphi''(t)&= -2 G^1(\varphi(t), y(\varphi(t)), \varphi'(t), \dot y(\varphi(t))\varphi'(t))
	\\
	\ddot y(\varphi(t)) \varphi'(t)^2+ \dot y(\varphi(t))\varphi''(t)
	&=\frac{d^2 (y(\varphi(t)))}{dt^2}&
	= -2G^2(\varphi(t), y(\varphi(t)), \varphi'(t), \dot y(\varphi(t))\varphi'(t))
	 
\end{array}.$$
	By $2$-homogenity of $G^i$ we see that the $x$-reparametrizations of geodesics with $\dot x \not=0$ are given as the solutions to the 2nd order ODEs
	\begin{equation}
	\begin{matrix}\label{inducedODE}
	 (X_+) & \ddot y = 2 G^1(x,y,+1,+\dot y) \dot y -2G^2(x,y,+1,+\dot y)& (\dot x >0) \\
	 (X_-) &\ddot y = 2G^1(x,y,-1,-\dot y) \dot y - 2 G^2(x,y,-1,-\dot y) & (\dot x <0)
	\end{matrix}.
	\end{equation}
	
	Note that the \textit{induced ODEs} (\ref{inducedODE}) determine the spray $\Gamma$ up to projective equivalence. Furthermore, the flow of every projective vector field of $\Gamma$ preserves this equations and is an infinitesimal point symmetry.
	If a spray admits a $3$-dimensional algebra of point symmetries, then the induced ODE must be (up to coordinate change) one from Lemma \ref{lemma1}.

	Let $\Gamma$ be a spray with $\dim \mathfrak p >3$. Then there are local coordinates where both induced ODEs (\ref{inducedODE}) must be of the form $y''=0$  and hence $\Gamma$ is projectively related to the flat spray $u \partial_x + v \partial_y$.
	 
	Let $\Gamma$ be a spray with $\dim \mathfrak p = 3$. Then we might assume that after a coordinate change $\mathfrak p$ is one of the in section \ref{sectionLieAlgebrasVectorFields} obtained Lie algebras of vector fields and  the two induced ODEs (\ref{inducedODE})
	\begin{equation}\label{inducedODEsGeneralForm}\begin{matrix}
	 (X_+) & \ddot y = f_+(x,y,\dot y) \\
	 (X_-) &\ddot y = f_-(x,y,\dot y) 
	\end{matrix}.\end{equation}
	have the corresponding form from Lemma \ref{lemma1} with possibly different constants $C_+$ and $C_-$.
		
	To understand whether a system (\ref{inducedODEsGeneralForm}) is induced by spray, we associate to a spray two more ODEs for its geodesics parametrized by the parameter $y$.
	$$\label{inducedODEsGeneralForm2}\begin{matrix}
	 (Y_+) & \ddot x = g_+(x,y,\dot x) \\
	 (Y_-) &\ddot x = g_-(x,y,\dot x) 
	\end{matrix}.$$
	By a similar calculation as for (\ref{inducedODE}), one finds that they are given as
	$$
	\begin{matrix}\label{inducedODE2}
	 \ddot x = 2 G^2(x,y,+\dot x,+1) \dot x -2G^1(x,y,+\dot x,+1)& (\dot y >0) \\
	 \ddot x = 2G^2(x,y,-\dot x,-1) \dot x - 2 G^1(x,y,-\dot x,-1) & (\dot y <0)
	\end{matrix}.
	$$
	and hence
	$$
	g_\pm(x,y,z) = \left\{ \begin{matrix}
	-z^3 f_\pm(x,y,\frac1z) & \text{if }z\geq 0\\
	-z^3 f_\mp(x,y,\frac1z)& \text{if }z\leq 0
	\end{matrix} \right..
	$$
	The function $f_{\pm},g_\pm$ must be defined and smooth at least on $U \times \mathbb R$ for some open subset $U \subseteq \mathbb R^2$ containing the origin. This excludes several possible ODEs:

	For $D1$, already $f_\pm(x,y,z)$ are not defined for all $z$.
	
	For $D2$ we have $f_+(x,y,z)=Cz^k$. If $k \in \{0,1,2,3\}$ again the assumptions of Lemma \ref{lemmaWhenFlat} are fulfilled. If $k \not=0,1,2,3$, $f_+$ or its $z$-derivatives have a singularity at $z=0$ unless $k$ is a natural number. But then $g_+(x,y,z)=-Cz^{-k+3}$ for $z \geq 0$ has a necessary singularity at $z=0$ and the same holds for $J1$. 
	
	For $J2$ we have $g_+(x,y,z)=-\tfrac12 z^2-C_+e^{-2x}$ for $z \geq 0$ and $g_+(x,y,z)=-\tfrac12 z^2-C_-e^{-2x}$ for $z  \leq 0$, so that $C_+=C_-$. Furthermore by the coordinate change $(x,y)\mapsto (x,\sqrt{2|C|} y)$ we can assume $C=\pm \frac{1}{2}$, and this ODEs are the induced ODEs of the sprays $(c^\pm)$ respectively.
	
	For $C1$ by evaluating $g_+$ and $g_-$ in $z=0$ one finds $C_-=-C_+ e^{\pi \lambda}$ and $C_+=-C_- e^{\pi \lambda}$, which is only possible if $\lambda=0$. By the coordinate change $(x,y) \mapsto (Cx,Cy)$ we can then assume $C_+ = -C_- = - 1$ and the ODEs are exactly the induced ODEs of the spray $(a)$.	
	
	For $C2$ similarly we find $C_-=-C_+$ and by $(x,y) \mapsto (x,-y)$ we can assume $C >0$. The ODEs are exactly the ones induced by the spray $(b_k^\pm)$.
	
	Two projective classes of sprays from the Lemma can not be transformed into each other by a coordinate transformation: This is obvious when the structure of the projective algebra $\mathfrak p$ with isotropy subalgebra $\mathfrak p_0$ is not isomorphic. We only need to distinguish $(c^+)$ from $(c^-)$ and $(b^+_k)$ (and $b^-_k$ respectively) for different $k>0$. One can either do this by direct calculations or using invariants for the induced ODEs, see \cite{doubrov}.

\section{Construction of the Finsler metrics}\label{SectionConstruction}	
	In this section we finish the proof of Theorem \ref{theorem1} by showing that the geodesic sprays of the Finsler metrics $(a,b^\pm_k,c^\pm)$ are projectively equivalent to the sprays from Lemma \ref{lemma2}. That each metric is not isometric to any projectively equivalent to one of the others follows from the additional statement of Lemma \ref{lemma2}.
	
	\begin{definition}\mbox{}
		\begin{enumerate}
			\item A strictily convex Finsler metric is a smooth function $F: T \mathbb R^2\backslash 0 \to \mathbb R$
				with the following properties:
				\begin{enumerate}
				\item $F(x,\xi)>0$ for all $(x,\xi) \in T\mathbb R^2 \backslash 0$ \label{positivity}
				\item $F(x,\lambda \xi) = \lambda F(x, \xi)$ for all $\lambda > 0$.\label{homogenity}
				\item The matrix $g_{ij}(x, \xi) = (\frac 12 \frac{\partial^2 F^2}{\partial \xi^i \xi^j}(x,\xi))_{ij}$ is positive-definite for all $(x,\xi) \in T\mathbb R^2 \backslash 0$.\label{regularity}
				\end{enumerate}
			\item
				The geodesic spray of $F$ is the vector field $\Gamma_F = \xi^i \partial_{x^i}-2G^i(x,\xi) \partial_{\xi^i}$ on $T \mathbb R^2 \backslash 0$ with $G^i(x, \xi) = \frac{1}{4} g^{ij}\left(2 \frac{\partial g_{jk}}{\partial x^l} - \frac{\partial g_{kl}}{\partial x^j} \right)\xi^k \xi^l$, where $(g^{ij})$ is the matrix inverse of $(g_{ij})$.\\
				For a general Lagrangian $L: T\mathbb R^2 \to \mathbb R$ we denote by $E_i(L,c)=\partial_{x^i}L - \frac{d}{dt}(\partial_{\xi^i} L)=0, i=1,2$ its Euler-Lagrange equations on curves $c: [a,b] \to \mathbb R^2$.\\
				The geodesics of $F$ are the solutions of $E_i(\frac12 F^2,c)=0$.
				The integral curves of the geodesic spray and the geodesics correspond each other under projection and prolongation.			
			
			\item Two Finsler metrics $F, \tilde F$ are called projectively equivalent, if their geodesics sprays $\Gamma, \tilde \Gamma$ are projectively equivalent, that is if they have the same geodesics as oriented point sets.
		\end{enumerate}
	\end{definition}
%	It was shown in \cite{Lovas} that every Finsler metric in the sense defined above is necessarilly strictly convex, meaning that the matrix $(g_{ij})$ is positive definit everywhere.
	\begin{lem} The induced ODEs (\ref{inducedODE}) of the geodesic spray of a Finsler metric $F$ are given by
	$$\begin{matrix}
	\ddot y = \left.\frac{\frac{\partial F}{\partial y}-\frac{\partial^2 F}{\partial x \partial v}-v \frac{\partial^2 F}{\partial y \partial v} }{\frac{\partial^2 F}{\partial v \partial v}}\right|_{(x,y,1,\dot y)} & \dot x >0 \\ 
	\ddot y = \left.\frac{\frac{\partial F}{\partial y}-\frac{\partial^2 F}{\partial x \partial v}-v  \frac{\partial^2 F}{\partial y \partial v} }{\frac{\partial^2 F}{\partial v \partial v}}\right|_{(x,y,-1,-\dot y)}& \dot x <0
	\end{matrix}.$$
	\end{lem}	
	\begin{proof}
	If $c$ is a geodesic of $F$, then it is an extremal value of the functional $c \mapsto \int_a^b \frac12 F^2(c, \dot c) dt$ and a solution to $E_i(\tfrac12 F^2, c)=0$. Then by positive homogenity of $F$ any orientation preserving reparametrization is an extremal of the length functional of $F$ and a solution to $E_i(F,c)=0$. 
	
	Let $\tilde c(t) = (t,\tilde y(t))$ or $\tilde c(t) = (-t,\tilde y(-t))$ be an orientation preserving reparametrization of $c$. Then the Euler-Lagrange equation $E_2(F,\tilde c)=0$ is satisfied, i.e.
	$\frac{\partial F}{\partial y} - \frac{\partial F}{\partial v \partial x} \dot x  - \frac{\partial F}{\partial v \partial y}  \dot y - \frac{\partial F}{\partial v \partial u} \ddot x - \frac{\partial F}{\partial v \partial v} \ddot y = 0$. Substituting the curve $\tilde c$ gives the equations.
	\end{proof}
	We now can calculate easily the induced ODEs of the Finsler metrics from Theorem \ref{theorem1} and see that they coincide with the ones induced by the sprays from \ref{lemma2}. In the following we explain how the Finsler metrics were constructed.

	\subsection{Riemannian metrics}\label{RiemannianSection}
	For a Riemannian metric the induced ODEs (\ref{inducedODE}) are of the form
	\begin{equation}\label{riemannianODE}
		\ddot y=K^0(x,y)+K^1(x,y)\dot y+K^2(x,y)\dot y^2+K^3(x,y) \dot y^3,
		\end{equation}
	so that only the two sprays $(c^\pm)$ can be projectively equivalent to the geodesic spray of a Riemannian metric.
	\begin{lem*}[\cite{matveev}]
	The induced ODEs (\ref{inducedODE}) of a pseudo Riemannian metric $g=(g_{ij})$ are given by (\ref{riemannianODE}),	
	if and only if the coefficients of the matrix $a = (\det g)^{-2/3} g$ satisfy the linear PDE system
	\begin{equation}\label{liouville}
	\left.\begin{array}{rrl}
	\partial_x a_{11}-\frac{2}{3}K^1a_{11}+2K^0a_{12}&=&0\\
	\partial_y a_{11}+2\partial_xa_{12}-\frac{4}{3}K^2a_{11}+\frac{2}{3}K^1a_{12}+2K^0a_{22}&=&0\\
	2\partial_ya_{12}+\partial_xa_{22}-2K^3a_{11}-\frac{2}{3}K^2a_{12}+\frac{4}{3}K^1a_{22}&=&0\\
	\partial_ya_{22}-2K^3a_{12}+\frac{2}{3}K^2 a_{22}&=&0
	\end{array}\right\}.\end{equation}
	\end{lem*}
	Note that one can reconstruct $g$ from $a$ by $g=\frac{1}{(\det a)^2}a$. Using the above Lemma, one can describe explicitly the $4$-dimensional space of psuedo-Riemannian metrics whose geodesic spray is projectively related to the spray $(c^\pm)$, in particular one finds the two locally Riemannian metrics	
	$$(c^+) ~\sqrt{ \frac{e^{3x}}{(2 e^x - 1)^2}dx^2 + \frac{e^x}{2e^x-1} dy^2  }~~~~~
	(c^-) ~\sqrt{ e^{3x}dx^2 + e^x dy^2}$$		
		whose geodesic spray is projectively equivalent to $(c^+)$ and $(c^-)$ respectively.

	\subsection{Randers metrics}\label{randersMetrics}
		Recall that a \textit{Randers metric} $F=\alpha+\beta$ is given as the sum of a Riemannian norm and a 1-form, i.e. $F(x,\xi) := \sqrt{\alpha_x(\xi,\xi)} + \beta_x(\xi)$, where $\alpha$ is a Riemannian metric and $\beta$ a 1-form.

		We now explain how to construct Finsler metrics whose geodesic spray is projectively equivalent to the remaining $(a)$ and $(b_k^\pm)$. Starting with a Riemannian metric $\alpha$ of constant sectional curvature and hence with $3$-dimensional Killing algebra $\mathfrak{iso}(\alpha)$, we construct a Randers metric $F$ whose geodesics are curves of constant geodesic curvature $k$ with respect to $\alpha$. Their geodesic spray will be projectively equivalent to $(a)$ (for the Euclidean metric and $k=1$), to $(b_k^+)$ for the standard metric on the two-sphere (in stereographic coordinates) and to $(b_k^-)$ for the metric of the hyperbolic plane (in the Poincare disk model).
		
		Let $\alpha=\alpha_{ij}dx^i dx^j$ be a Riemannian metric on the plane and choose a multiple of the volume form $\Omega_x = -k \sqrt{\det \alpha_x}\, dx^1 \wedge dx^2$ with $k > 0$. Since we may work on a simply connected neighborhood, we can choose a $1$-form $\beta=\beta_j dx^j$, whose exterior derivative is $\Omega$.

		\begin{lem}\label{construction}
		The projective algebra of the Randers metric $F(x,\xi):= \sqrt {\alpha_x(\xi,\xi)} + \beta_x(\xi)$ contains the Killing algebra of $\alpha$. Its geodesics are exactly the positively oriented curves of constant geodesic curvature $k$ with respect to $\alpha$.
		\end{lem}		
		\begin{proof}
		Both $\alpha$ and $\Omega$ induce a natural bundle isomorphism $\phi_\alpha,\phi_\Omega: T\mathbb R^2 \to T^*\mathbb R^2$, which in coordinates are given by $(x,\xi) \mapsto (x,\alpha \xi)$ and $(x, \xi) \mapsto (x, \Omega\xi)$, where $\alpha=(\alpha_{ij})$ and $\Omega=(\Omega_{ij})$ are the Gramian matrices wrt. the fixed coordinates. The map $J:=\phi_\alpha^{-1} \circ \phi_\Omega$ is given by $\xi \mapsto g^{-1}\Omega \xi$ and is a bundle autmorphism $T\mathbb R^2 \to T \mathbb R^2$ with $J^2=-k^2\text{Id}$. Indeed, it is $\Omega \alpha^{-1} \Omega = -k^2 \det (\alpha) \det(\alpha^{-1}) (\alpha^{-1})^{-1} = -k^2 \alpha$.
		
		Consider the Euler-Lagrange equations for $L(x,\xi)=\frac12 \alpha_x(\xi,\xi)+\beta_x(\xi)$.
		We first calculate $E_i(\beta,c)= \frac{\partial \beta_j}{\partial x^i} \dot c^j -\frac{\partial \beta_i}{\partial x^j} \dot c^j = (\Omega \dot c)_i$. Contracting the equations $E_j(L,c)=0$ with $\alpha^{ij}$ and using the Levi-Civita connection $\nabla$ of $\alpha$, gives
		$$\alpha^{ij}E_j(L) = \alpha^{ij}E_i(\tfrac{1}{2}\alpha, c)+\alpha^{ij}E_i(\beta,c)=-(\nabla_{\dot c} \dot c)^i + \alpha^{ij}\Omega_{jk}\dot c ^k,$$
		so that the system $E_i(L,c)=0$ is equivalent to
		\begin{equation}\label{magneticGeodesics}
		\nabla_{\dot c} \dot c = \alpha^{-1}\Omega \dot c = J \dot c.
		\end{equation}
		
		This equation is sometimes referred to as the equation for \textit{magnetic geodesics} for the Lorentz force
		$J: T\mathbb R^2 \to T \mathbb R^2$.
		Note that the solutions $c$ of equation (\ref{magneticGeodesics}) have constant $\alpha$-velocity, since $\frac12\frac{d}{dt} \alpha(\dot c, \dot c)=\alpha( \dot c,\nabla_{\dot c} \dot c)=\alpha(\dot c, J\dot c)=\Omega(\dot c, \dot c)=0$.
		Furthermore the equation is preserved under $\alpha$-isometries and the flow of any Killing vector field takes each $\alpha$-unit speed solution to a $\alpha$-unit speed solution.
		
		Now consider the Euler-Langrange equation of the Randers metric $F(x,\xi)=\sqrt{\alpha_x(\xi,\xi)}+\beta_x(\xi)$. We claim that its geodesics reparametrized $g$-unit speed solutions of equation (\ref{magneticGeodesics}). Indeed, if $c$ has this properties, then
		$$E_i(\alpha^2,c)=2{\alpha}(\partial_{x^i}{\alpha})- \frac{d}{dt}\Big(2 \alpha (\partial_{\xi^i}{\alpha})\Big)=2E_i( \alpha,c)$$ and hence
		$$E_i(F,c)=E_i( \alpha,c) + E_i(\beta,c)=E_i(\frac12 \alpha^2,c) + E_i(\beta,c)=E_i(L,c)=0.$$ Since the family of solutions to $E_i(F,c)=0$ are exactly all orientation preserving reparametrizations of the solutions of $E_i(\frac12 F^2,c)=0$, we see that the geodesics of $F$ are exactly the $\alpha$-unit speed solutions to equation (\ref{magneticGeodesics}) reparametrized to $F$-arc length. In particular every isometry of $\alpha$ preserves the geodesics of $F$ as oriented point sets and we have $\mathfrak{iso}(\alpha) \subseteq \mathfrak p(F)$.
		
		The geodesics have constant geodesic curvature $\kappa_\alpha=k$, since for their $\alpha$-unit speed parametrization $c$ we have
		$$\kappa_\alpha(c)^2=\alpha(\nabla_{\dot c}\dot c,\nabla_{\dot c}\dot c)
		=\alpha(J \dot c, J\dot c)=- \dot c^t \Omega \alpha^{-1} \Omega \dot c = k^2 \cdot \alpha(\dot c, \dot c)=k^2.$$
		\end{proof}

	To produce Randers metrics with three dimensional projective algebra, we may choose $g$ as a metric of constant curvature, since each of them admit three independent Killing vector fields.
	For the Euclidean metric $dx^2+dy^2$ and volume form $-dx \wedge dy$, we might choose $\beta = \frac{1}{2}(y dx - x dy)$.
	For the sphere/hyperbolic metric $\frac{dx^2+dy^2}{(1\pm(x^2+y^2))^2}$ and volume form $\frac{-k}{(1\pm(x^2+y^2))^2} dx \wedge dy$, we might choose $\beta=\frac{k}{2} \frac{y dx - x dy}{1\pm(x^2+y^2)}$.
	
	 The result are the Randers metrics $(a)$ and $(b_k^\pm)$ from Theorem \ref{theorem1}. By Lemma \ref{construction}, we have $\mathfrak{iso}(\alpha) \subseteq \mathfrak p (F)$. In fact we have $\mathfrak{iso}(\alpha) = \mathfrak p (F)$: otherwise $\dim \mathfrak p (F) >3$  and $F$ would be projectively equivalent to the Euclidean metric by Lemma \ref{lemma2} and in particular geodesically reversible, which is obviously not the case.

	\bibliographystyle{plain}
	\bibliography{literature}

\begin{thebibliography}{10}

\bibitem{AlvarezProfFlat1}
J.~\'{A}lvarez Paiva.
\newblock Hilbert's fourth problem in two dimensions.
\newblock In {\em M{ASS} selecta}, pages 165--183. Amer. Math. Soc.,
  Providence, RI, 2003.

\bibitem{AlvarezProfFlat2}
J.~\'{A}lvarez Paiva.
\newblock Symplectic geometry and {H}ilbert's fourth problem.
\newblock {\em J. Differential Geom.}, 69(2):353--378, 2005.

\bibitem{alvarez}
J.~\'{A}lvarez Paiva and G.~Berck.
\newblock Finsler surfaces with prescribed geodesics.
\newblock {\em ArXiv e-prints}, February 2010.

\bibitem{Arcostanzo}
M.~Arcostanzo.
\newblock Des m\'etriques finsl\'eriennes sur le disque \`a partir d'une
  fonction distance entre les points du bord.
\newblock In {\em S\'eminaire de {T}h\'eorie {S}pectrale et {G}\'eom\'etrie,
  {N}o.\ 10, {A}nn\'ee 1991--1992}, volume~10 of {\em S\'emin. Th\'eor. Spectr.
  G\'eom.}, pages 25--33. Univ. Grenoble I, Saint-Martin-d'H\`eres, 1992.

\bibitem{matveev}
R.~Bryant, G.~Manno, and V.~Matveev.
\newblock A solution of a problem of {S}ophus {L}ie: normal forms of
  two-dimensional metrics admitting two projective vector fields.
\newblock {\em Math. Ann.}, 340(2):437--463, 2008.

\bibitem{overviewFlat}
X.~Cheng, X.~Ma, Y.~Shen, and S.~Liu.
\newblock Hilbert's fourth problem and projectively flat finsler metrics.
\newblock In {\em Lie groups, differential equations, and geometry}, UNIPA
  Springer Ser., pages 247--263. Springer, Cham, 2017.

\bibitem{doubrov}
B.~Doubrov and B.~Komrakov.
\newblock The geometry of second-order ordinary differential equations, 2016.
\newblock arXiv:1602.00913.

\bibitem{lie}
S.~Lie.
\newblock Untersuchungen \"uber geod\"atische {C}urven.
\newblock {\em Math. Ann.}, 20(3):357--454, 1882.
\newblock Abschn. I, Nr. 4, Problem II.

\bibitem{LieODEs}
S.~Lie.
\newblock {\em Vorlesungen über Differentialgleichungen mit bekannten
  infinitesimalen Transformationen}.
\newblock Teubner, Leipzig, 1891.

\bibitem{randersRigidity}
V.~Matveev.
\newblock On projective equivalence and pointwise projective relation of
  {R}anders metrics.
\newblock {\em Internat. J. Math.}, 23(9):1250093, 14, 2012.

\bibitem{ShenProjFlat3}
X.~Mo, Z.~Shen, and C.~Yang.
\newblock Some constructions of projectively flat {F}insler metrics.
\newblock {\em Sci. China Ser. A}, 49(5):703--714, 2006.

\bibitem{romanovski}
Yu. Romanovski\u{\i}.
\newblock Calculation of local symmetries of second-order ordinary differential
  equations by {C}artan's equivalence method.
\newblock {\em Mat. Zametki}, 60(1):75--91, 159, 1996.

\bibitem{ShenModernFinsler}
Y.~Shen and Z.~Shen.
\newblock {\em Introduction to modern {F}insler geometry}.
\newblock Higher Education Press, Beijing; World Scientific Publishing Co.,
  Singapore, 2016.

\bibitem{ShenProjFlat1}
Z.~Shen.
\newblock Projectively flat {F}insler metrics of constant flag curvature.
\newblock {\em Trans. Amer. Math. Soc.}, 355(4):1713--1728, 2003.

\bibitem{ShenProjFlat2}
Z.~Shen.
\newblock On projectively flat {$(\alpha,\beta)$}-metrics.
\newblock {\em Canad. Math. Bull.}, 52(1):132--144, 2009.

\bibitem{Tabachnikov}
S.~Tabachnikov.
\newblock Remarks on magnetic flows and magnetic billiards, {F}insler metrics
  and a magnetic analog of {H}ilbert's fourth problem.
\newblock In {\em Modern dynamical systems and applications}, pages 233--250.
  Cambridge Univ. Press, Cambridge, 2004.

\bibitem{Tresse2}
A.~Tresse.
\newblock Determination des invariants ponctuels de l'equation differentielle
  ordinaire du second ordre y'' = w(x, y, y').

\bibitem{Tresse}
A.~Tresse.
\newblock Sur les invariants diff\'erentiels des groupes continus de
  transformations.
\newblock {\em Acta Math.}, 18(1):1--3, 1894.

\end{thebibliography}

\end{document}